\documentclass{amsart}[10pt]

\usepackage{amssymb,amsmath,amscd, amsthm,stmaryrd,verbatim,
mathrsfs,tikz,color , amsfonts, 
}
\usetikzlibrary{patterns}

\usepackage[colorlinks,linkcolor=blue,urlcolor=cyan,citecolor=green,pagebackref]{hyperref}

\numberwithin{equation}{section}
\allowdisplaybreaks

\usepackage[vcentermath,enableskew,noautoscale]{youngtab}
\usepackage{young}

\newtheorem{Thm}{Theorem}[section]
\newtheorem{Prop}[Thm]{Proposition}
\newtheorem{Lem}[Thm]{Lemma}
\newtheorem{Cor}[Thm]{Corollary}

\newtheorem{Defff}[Thm]{Definition}
\newtheorem{ex}[Thm]{Example}

\numberwithin{equation}{section}

\newcommand{\ms}{\mathscr}
\newcommand{\mr}{\mathrm}

\newcommand{\al}{\alpha}

\newcommand{\op}{\oplus}

\newcommand{\eq}{\begin{equation}}
\newcommand{\en}{\end{equation}}
\newcommand{\beqna}[1]{\begin{eqnarray}\label{#1}}
\newcommand{\eeqna}{\end{eqnarray}}
\newcommand{\beqn}[1]{\begin{equation}\label{#1}}
\newcommand{\eeqn}{\end{equation}}

\newcommand{\gk}{Gelfand-Kirillov dimension }
\newcommand{\gks}{Gelfand-Kirillov dimensions }
\newcommand{\af}{$ \mathbf{a} $}
\newcommand{\mc}[1]{\mathcal{#1}}
\newcommand{\msc}[1]{\mathscr{#1}}
\newcommand{\mf}[1]{\mathfrak{#1}}
\renewcommand{\subset}{\subseteq}
\newcommand{\rar}{\rightarrow}
\newcommand{\bil}[2]{\langle{#1},{#2}^{\vee} \rangle }
\newcommand{\hs}{ \mathfrak{h}^*}
\newcommand{\Ocat}{\mathscr{O}}
\newcommand{\aff}{ \mathbf{a} }
\newcommand{\gkd}{\operatorname{GKdim}}
\renewcommand{\hom}{\operatorname{Hom}}
\newcommand{\sln}{\mathfrak{sl}(n)}
\newcommand{\lambdarho}{\lambda+\rho=(\lambda_1,\lambda_2,\cdots,\lambda_n)}
\newcommand{\sn}{\mf{S}_n}
\newcommand{\su}{{SU}}

\numberwithin{equation}{section}

\newlength\cellsize 
\savebox2{%
\begin{picture}(15,15)
\put(0,0){\line(1,0){15}}
\put(0,0){\line(0,1){15}}
\put(15,0){\line(0,1){15}}
\put(0,15){\line(1,0){15}}
\end{picture}}
\newcommand\cellify[1]{\def\thearg{#1}\def\nothing{}%
\ifx\thearg\nothing
\vrule width0pt height\cellsize depth0pt\else
\hbox to 0pt{\usebox2\hss}\fi%
\vbox to 15\unitlength{
\vss
\hbox to 15\unitlength{\hss$#1$\hss}
\vss}}
\newcommand\tableau[1]{\vtop{\let\\=\cr
\setlength\baselineskip{-16000pt}
\setlength\lineskiplimit{16000pt}
\setlength\lineskip{0pt}
\halign{&\cellify{##}\cr#1\crcr}}}
\savebox3{%
\begin{picture}(15,15)
\put(0,0){\line(1,0){15}}
\put(0,0){\line(0,1){15}}
\put(15,0){\line(0,1){15}}
\put(0,15){\line(1,0){15}}
\end{picture}}
\newcommand\expath[1]{%
\hbox to 0pt{\usebox3\hss}%
\vbox to 15\unitlength{
\vss
\hbox to 15\unitlength{\hss$#1$\hss}
\vss}}
\newcommand\bas[1]{\omit \vbox to \cellsize{ \vss \hbox to \cellsize{\hss$#1$\hss} \vss}}

\usepackage{hyperref}

\begin{document}

\bibliographystyle{plain}

\title[GK dimensions of certain Harish-Chandra modules]{Gelfand-Kirillov Dimensions of  Highest Weight Harish-Chandra Modules for $\su(p,q)$}

\author{Zhanqiang Bai}\author{Xun Xie}

\address{School of Mathematics and Statistics, Wuhan University, Wuhan 430072, China}
\email{Zhanqiang.bai@whu.edu.cn}

\address{School of Mathematics and Statistics, Beijing Institute of Technology, Beijing 100081, China}
\email{xieg7@163.com}



\date{\today}



\maketitle

\begin{abstract}
Let $ (G,K) $ be an irreducible Hermitian symmetric pair of non-compact type with  $G=\su(p,q)$,  and let $ \lambda $ be an integral weight such that the simple highest weight module $  L(\lambda) $ is a Harish-Chandra $ (\mathfrak{g},K) $-module.  We give a combinatorial  algorithm for the Gelfand-Kirillov dimension of  $ L(\lambda) $. This enables us to  prove that the Gelfand-Kirillov dimension of  $ L(\lambda) $  decreases  as the integer $ \bil{\lambda+\rho}{\beta} $  increases, where $\rho$ is the half sum of positive roots and $\beta$ is the maximal noncompact root. Finally by the combinatorial algorithm, we obtain  a  description of the associated variety of $ L(\lambda) $. 
\end{abstract}

\section{Introduction}

In the representation theory of Lie groups and Lie algebras, the Gelfand-Kirillov dimension is an important invariant to measure the size of   an infinite-dimensional module. This kind of invariant was first introduced by Gelfand and Kirillov \cite{Ge-Ki}. In this paper, we are concerned with the Gelfand-Kirillov dimensions of highest weight $ (\mf{g},K) $-modules, where $ \mf{g} $ is the Lie algebra of  the Hermitian type Lie group $G= \su(p,q) $ and $ K $  is the maximal compact subgroup $ K=S(U(p)\times U(q)) $ with   $n=p+q$.

To state our problem, we need some notations  from  \cite{EHW}.
Let $ (G,K) $ be an irreducible Hermitian symmetric pair of non-compact type. We denote their	 Lie algebras by ($\mathfrak{g}_{0}$, $\mathfrak{k}_{0}$), and denote their complexification by $\mathfrak{g}={\mathfrak{g}_{0}}\bigotimes \mathbb{C}$ and $\mathfrak{k}=\mathfrak{k}_{0}\bigotimes\mathbb{C}$. Then $\mathfrak{k}=\mathbb{C}H\bigoplus[\mathfrak{k}, \mathfrak{k}]$ with $\operatorname{ad}(H)$ having eigenvalues $0, 1, -1$ on $\mathfrak{g}$. If we denote $\mathfrak{p}^{\pm}=\{X\in \mathfrak{g}\mid [H, X]=\pm X\}$. Then $\mathfrak{g}=\mathfrak{p}^{-}\bigoplus \mathfrak{k}\bigoplus \mathfrak{p}^{+}$. Let $\mathfrak{h}_{0} \subseteq \mathfrak{k}_{0}$ be a Cartan subalgebra. 
 Let $\Phi$ denote the set of  roots with respect to ($\mathfrak{g},\mathfrak{h}$), let $ W $ denote the Weyl group of $ \mf{g} $, and let $\Phi_{c}$ (resp. $\Phi_{n}$) denote the set of compact roots (resp. noncompact roots). Choose a Borel subalgebra $\mathfrak{b}$ containing $\mathfrak{h}$. Let $\beta $ denote the unique maximal non-compact root of $\Phi^{+}$. Now choose $\zeta  \in \mathfrak{h}^{*} $ so that  $\zeta $ is orthogonal to $\Phi_{c}$  and $\bil{\zeta }{ \beta}=1$. For $ \lambda\in\hs $,
the simple highest weight module $ L(\lambda) $ is a $ (\mf{g},K) $-module if and only if $ \lambda $ is $ \Phi_c^+ $-dominant integral (i.e $ \lambda $ is dominant and integral as a $ \mf{k} $-weight) and $ \bil{\lambda+\rho}{\beta} \in\mathbb{R}$. In this paper, we  call such a module $ L(\lambda) $  a \textit{highest wight $ (\mf{g},K) $-module} for simplicity. Note that if $ L(\lambda) $ is a $ (\mf{g},K) $-module then we can write $ \lambda=\tilde{\lambda}+z\zeta $, 
for some  $ z\in\mathbb{R} $ and for an integral and $ \Phi_c ^+$-dominant weight $ \tilde{\lambda} $ such that $ \bil{\tilde{\lambda}+\rho}{\beta} =0$; in fact $ z=\bil{\lambda+\rho}{\beta} $ and $ \tilde{\lambda}=\lambda-z\zeta $.

The motivation of this paper origins from the study  on unitary  highest weight $ (\mf{g},K) $-modules; see \cite{Bai-Hu}.
In \cite{Bai-Hu},  the Gelfand-Kirillov dimension of a unitary highest weight module $ L(\lambda) $ is expressed explicitly in terms of $ z:=\bil{\lambda+\rho}{\beta} $. From this result,  one can see that the Gelfand-Kirillov dimension decreases as $ z $ takes unitary reduction points and increases. Then a natural question is to ask whether   a similar result holds  for all highest weight $ (\mf{g},K) $-modules.  We will obtain   a positive answer in this paper for the Lie group $ \su(p,q) $.

In 1978, Joseph \cite{Joseph-78} found that the GK-dimension of  a highest weight $\mathfrak{sl}(n)$-module $L_w:=L(-w\rho-\rho)$ with $w\in W$ can be computed by Robinson-Schensted correspondence. This result suggests a combinatorial algorithm for highest weight $ (\mf{g},K) $-modules in the case $ G=\su(p,q) $.
It is known that   $ \gkd L(\lambda) $  for any weight $ \lambda $ can be reduced to $L_w$ for some $ w\in W $ by using Jantzen's translation functors \cite{Jantzen79Lec750}, but we fail to find a concrete result in published literatures. In the first part of this paper, we will formulate an explicit algorithm  for $  \gkd L(\lambda)$ for $ \mf{g}=\mf{sl}(n) $  and any  weight $ \lambda\in\hs $. A main tool that we use is Lusztig's $ \aff $-function on the Weyl group.

The  $ \aff $-function  is  defined in a  combinatorial way using the Kazhdan-Lusztig basis of the Hecke algebra; see \cite{lusztig1985cellsI} or \S \ref{sec:cell}. There is a formula  connecting the $ \aff $-function and the Gelfand-Kirillov  dimension\begin{equation*}
\gkd L_w=|\Phi^+|-\aff(w) \text{ for } w\in W,
\end{equation*}
which is formulated in \cite{lusztig1985A_n}.
Thus computations of Gelfand-Kirillov  dimensions can be ultimately reduced to that of  Lusztig's $ \aff $-functions.
In the case $ \mf{g}=\mf{sl}(n) $, the Weyl group is isomorphic to the symmetric group $ \mf{S}_n $ in $ n $ letters, and the $ \aff $-functions on $ \mf{S}_n $ can be easily determined by using some known properties about two-side cells.  Our results can be stated as follows.

\begin{Prop}[See Prop. \ref{pr:main1}]
Let $ \mf{g} $ be a simple complex Lie algebra. For $ \lambda\in\hs $, let $ W_{[\lambda]} $ be the integral Weyl group relative to $ \lambda $, i.e the subgroup of $W$ generated by the set of reflections $\{s_{\alpha}\mid \langle \lambda+\rho,\alpha^{\vee}\rangle\in \mathbb{Z}\}$ and let $ \aff_{[\lambda]}:W_{[\lambda]}\rar \mathbb{N}$ be  the Lusztig's $ \aff $-function associated to $ W_{[\lambda]} $. Then we have \[
\gkd L(\lambda)=|\Phi^+|-\aff_{[\lambda]}(w)
\]where $ w\in W_{[\lambda]} $ is the element of minimal length such that $ w^{-1}.\lambda $  is antidominant (under the dot action).
\end{Prop}

Now we assume  that  $ \mf{g}=\mf{sl}(n) $. By the proposition above, we have an explicit algorithm of $ \gkd L(\lambda) $. For a weight $ \lambda\in\hs $ we write $ \lambdarho $; see \S \ref{sec:notation}. We associated to $ \lambda $ a set $ P(\lambda) $ of some Young tableaux as follows. Let $ \lambda_X:\lambda_{i_1}, \lambda_{i_2}, \dots, \lambda_{i_r} $  be a maximal subsequence of $ \lambda_1,\lambda_2,\dots,\lambda_n $ such that $ \lambda_{i_k} $, $ 1\leq k\leq r $ are congruent to each other modulo $ \mathbb{Z} $. Then the Young tableau associated to the subsequence $ \lambda_X $ using Schensted insertion algorithm (see \S\ref{sec:schensted}) is a Young tableau in $ P(\lambda) $. If $ \lambda $ is an integral weight then $ P(\lambda) $ consists of only one Young tableau and we identify $ P(\lambda) $ with this Young tableau in this case.

If $ Y $ is a Young tableau, we define $ A(Y) :=\sum_{i\geq 1}\frac{c_i(c_i-1)}2$ where $ c_i $ is the number of entries in the $ i $-th column of $ Y $. And we  define $ A(P(\lambda)) :=\sum_{Y\in P(\lambda)}A(Y)$.

\begin{Thm}[See Thm. \ref{thm:main2}]
Let $ \mf{g}=\mf{sl}(n) $. For any $ \lambda\in\hs $, we have \[
\gkd L(\lambda)=\frac{n(n-1)}2-A(P(\lambda)).
\]
\end{Thm}

\begin{ex}
If $ \lambda+\rho=(3,3.5,2,1.5,-1,5.5,-1,0,1.1) $, then $ P(\lambda) $ has three Young tableaux:
\[
 \tableau{-1&-1 & 0\\2\\3}\qquad\tableau{1.5&5.5\\3.5}\qquad\tableau{1.1}
\]
By the theorem, $ \gkd L(\lambda)=\binom{9}{2}-A(P(\lambda)) =\binom{9}{2}-\binom{3}{2}-\binom{2}{2}=36-3-1=32$.
\end{ex}


Set $ G=\su(p,q) $.  In this case, we have $ \rho=(-1,-2,\dots,-n) $, $ \beta=\epsilon_1-\epsilon_n $, and  $ \zeta=\sum_{i=1}^p\epsilon_i $.
 A necessary condition for $ L(\lambda) $ being  a $ (\mf{g},K) $-module  is that $ \lambda $ is a \textit{$ (p,q) $-dominant} weight, i.e. $ \lambda_i-\lambda_j\in\mathbb{Z}_{>0} $ for all $ 1\leq i<j\leq p $ and $ p+1\leq i<j\leq p+q $ where $ \lambdarho $.  Now we can concentrate on $\gkd L(\lambda) $ for a $ (p,q) $-dominant weight.

\begin{Thm}[See Thm. \ref{pr:pqw}] Assume that $ \lambda $ is  $ (p,q) $-dominant with $ \lambdarho $. Then
\begin{itemize}
\item [(i)] If $ \lambda_1-\lambda_{p+1}\in\mathbb{Z} $, i.e. $ \lambda $ is an integral weight, then $ P(\lambda) $ is a Young tableau with at most two columns. In this case we have $ \gkd L(\lambda)=m(n-m) $ where $ m $ is the number of entries in the second column of $ P(\lambda) $.
\item [(ii)] If $ \lambda_1-\lambda_{p+1}\notin\mathbb{Z} $, then $ P(\lambda) $ consists of two Young tableaux with single column. In this case we have $ \gkd L(\lambda)=pq $.
\end{itemize}
\end{Thm}

By this theorem, to determine the Gelfand-Kirillov  dimension of a highest weight $ (\mf{g},K) $-module, we only need to determine  the number $  m $ in (i).
We have
  a  combinatorial model 
  for computing $ m $ as follows. 
 
Let $ \lambda $ be an  integral and $ (p,q) $-dominant weight with $ \lambda+\rho=(\lambda_1,\cdots,\lambda_n) $. We associate to $ \lambda $ a line of white and black balls as follows.
Assume that we have $ n $ balls, which are filled with integers $ \lambda_i $, $ 1\leq i\leq n $. We paint a ball black (resp. white)
 if it contains $ \lambda_i $ for some $ 1\leq i\leq p $ (resp. $p+1\leq i\leq p+q=n  $). Then we arrange these balls in a line such that the entries in these balls weakly decrease and such that if $ \lambda_i=\lambda_j $ for some $ 1\leq i\leq p $, $ p+1\leq j\leq n $,   the white ball containing $ \lambda_j $ is on the left side of  the black ball containing $ \lambda_i $. 
We call two balls an adjacent white-black pair if they are adjacent with the left ball white and the right ball  black. We remove the  adjacent  white-black pairs from this line of balls over and over again until there are no adjacent white-black  pairs.
 
 \begin{Thm}[See Thm. \ref{pr:main4}]
 	Keep  notations as above. The number $ m $ of entries in the second column of $ P(\lambda) $ is precisely the number of adjacent white-black pairs that we can remove from this line of balls.
 \end{Thm}
 \begin{ex}
Let $ \lambda $ be a $ (5,5) $-dominant weight such that $$  \lambda+\rho=(5,4,3,2,1,9,8,7,6,2) . $$ Then the line of balls associated to $ \lambda $ is as follows:
\begin{center}
\begin{tikzpicture}

\node at ( 1.5,1) [circle,draw=black] {};
\node at ( 2,1) [circle,draw=black] {};
\node at ( 2.5,1) [circle,draw=black] {};

\node at ( 3,1) [circle,draw=black] {};
\node at ( 3.5,1) [circle,draw=black,fill=black!60]  {};

\node at ( 4,1) [circle,draw=black,fill=black!60] {};

\node at ( 4.5,1)[circle,draw=black,fill=black!60] {};

\node at ( 5,1) [circle,draw=black]  {};

\node at ( 5.5,1)[circle,draw=black,fill=black!60] {};

\node at ( 6,1) [circle,draw=black,fill=black!60] {};
\end{tikzpicture}
\end{center}
The number of adjacent white-black pairs that we  can remove from this line of balls is 5. Hence $ m=5 $, and $ \gkd L(\lambda) =5\times (10-5)=25$.
 \end{ex}

 From this combinatorial model, we can reprove a result of \cite{Bai-Hu} on Gelfand-Kirillov dimensions of  unitary highest weight modules (in the case $ G=\su(p,q) $); see Prop. \ref{pro:uhw}. Furthermore, we can prove the following result.

\begin{Thm}[See Thm. \ref{th:last}]\label{th:intr2}
Assume that $ G=SU(p,q) $ and  $ \tilde{\lambda} $ is a $ (p,q) $-dominant weight such that $ \bil{\tilde{\lambda}+\rho}{\beta} =0$. Then $\gkd L(\tilde{\lambda}+z\zeta) $ deceases to 0 as $ z $ increases in $ \mathbb{Z} $. And $ \gkd L(\tilde{\lambda}+z\zeta)=pq $ if $ z\in \mathbb{R}\setminus\mathbb{Z} $.
\end{Thm}

 
 From Vogan \cite{Vogan-91}, we know that the associated variety of any highest weight $(\mathfrak{g},K)$-module is the closure of some $K_{\mathbb{C}}$ orbit in $\mathfrak{p}^+\cong (\mathfrak{g}/(\mathfrak{k}+\mathfrak{p}^+))^*$. The closures of  $K_{\mathbb{C}}$ orbits in $ \mathfrak{p}^+ $ form a linear chain of varieties:
 \begin{equation}\label{eq:chain}
 \{0\}={\bar{\mathcal{O}}}_0\subset \bar{\mathcal{O}}_1\subset ...\subset\bar{\mathcal{O}}_{r-1}\subset \bar{\mathcal{O}}_r=\mathfrak{p}^+,
 \end{equation}
 where $r$ is the rank of the Hermitian symmetric space $G/K$. In the case of 	$ G=\su(p,q) $, $ r=\min\{p,q\} $.
\begin{Thm}[See Thm. \ref{th:last2}] If $ G=\su(p,q) $ and $\lambda$ is an integral weight such that   $ L(\lambda)$ is a highest weight $(\mf{g},K)$-module, then the associated variety of $ L(\lambda)$  is the $ m $-th closure $\bar{\mathcal{O}}_m$  in \eqref{eq:chain}, where $ m $ is the number of entries in the second column of the Young tableau $ P(\lambda) $.
\end{Thm}

As pointed out by one of our  referees,  Garfinkle \cite{Garfinkle1993} also obtained    Gelfand-Kirillov dimensions and associated varieties of irreducible Harish-Chandra modules with trivial infinitesimal character for $ SU(p,q) $ by a different method. One of   advantages of our method in this paper is that we can uniformly and  directly obtain these invariants  from the highest weight of  any simple highest weight Harish-Chandra module. We thank the referee for pointing out  Garfinkle's work.

\section{Preliminaries on Gelfand-Kirillov Dimension}
In this section we recall the  definition and some properties of the Gelfand-Kirillov dimension.  The details can be found in \cite{Bo-Kr,Ja,Kr-Le,NOTYK,Vogan-78,Vogan-91}.

\begin{Defff}
Let $A$ be an algebra   generated  by a
finite-dimensional subspace $V$. Let $V^n$ denote the linear span of
all products of length   at most $n$ in elements of $V$. The
 {Gelfand-Kirillov dimension} of $A$  is defined by:
$$\gkd(A) =  \limsup_{n\rightarrow \infty} \frac{\log\mathrm{dim}( V^{n} )}{\log n}$$
\end{Defff}

 It is well-known that the above definition  is independent of  the choice of the
finite  dimensional  generating subspace $V$ (see \cite{Bo-Kr,Kr-Le}).
Clearly $\gkd(A)=0$ if and only if $\mathrm{dim}(A) <\infty$.
 
The notion of Gelfand-Kirillov dimension can be extended for left $A$-modules. In fact, we have the following definition.
\begin{Defff} Let $A$ be  an algebra   generated  by a finite-dimensional subspace $V$. Let $M$ be a left $A$-module generated  by a finite-dimensional subspace $M_{0}$.   
The {Gelfand-Kirillov dimension} $\gkd(M)$ of $M$  is defined by
$$\gkd(M) = \limsup_{n\rightarrow \infty}\frac{\log\mathrm{dim}( V^{n}M_{0} )}{\log n}.$$
\end{Defff}

%
%

\begin{Lem} Let $ \mf{g} $ be a finite dimensional Lie algebra, $M$ be a finitely generated $U(\mathfrak{g})$-module and $F$ be a finite dimensional  $U(\mathfrak{g})$-module. Then we have
	$\gkd(M\otimes F)=\gkd(M)$.
\end{Lem}


\section{\gks and Lustig's \af-functions}

\subsection{Translation functors}
We first recall   the category $ \msc{O} $ of a semisimple Lie algebra; see for example \cite{Hum08}.
Let $ \mf{g} $  be a semisimple Lie algebra over $ \mathbb{C} $  and  $ \mf{h} $  be a Cartan subalgebra. Let $ \Phi \subset \mf{h}^*$ be the root system of $ \mf{g} $ with respect to $ \mf{h} $.   Choose a set  $ \Phi^ +$ of positive roots. Let $ \Delta $ be the set of  simple roots  and let  $ \Phi^-=-\Phi^+ $. We have a triangular decomposition
\(\mf{g}=\mf{n}^-\oplus\mf{h}\oplus\mf{n}^+\)
where $ \mf{n}^+=\bigoplus_{\alpha\in\Phi^+}\mf{g}_\alpha $ and $ \mf{n}^-=\bigoplus_{\alpha\in\Phi^-}\mf{g}_\alpha $. Let $ \mf{b}^+=\mf{h}\oplus\mf{n}^+ $ (resp. $ \mf{b}^- =\mf{h}\op \mf{n}^-$) be the positive (resp. negative) Borel
subalgebra with respect to $ \Delta $.

Let $ (\,,\,):\mf{h}^*\times\mf{h}^* \rar \mathbb{C} $ be the bilinear form such that for all $ \al\in \Phi^+ $, $ \lambda(\al^{\vee})=\frac{2(\lambda,\al)}{(\al,\al)} $,
 where $ \lambda\in\mf{h}^* $ and $ \al^{\vee}\in\mf{h} $ is the coroot of $ \al $. As usual, we   write $\bil{\lambda}{\al}$ for $ \lambda(\al^{\vee}) $.
 The Weyl group $ W $ is generated by reflections $ s_\al $, $ \al\in\Phi $ where $ s_\al $ acts on $ \hs$ by \[
 s_\al(\lambda)=\lambda-\bil{\lambda}{\al}\al\quad\text{ for all } \lambda\in\hs.
 \]Then $ W $ is a Coxeter group with  $ S=\{s_\al\mid \al\in\Delta \} $ being the set of generators. Let $ \rho $ be the half sum of all positive roots. The shifted action \[
 w.\lambda:=w(\lambda+\rho)-\rho, w\in W,\lambda\in\hs
 \] is called the \textit{dot action} of $ W $ on $ \hs $. If we denote $ w_0 $ by the longest element of $ W $, then
 $ w_0.0=-2\rho $, and for any $ w\in W $ we have $ w.(-2\rho)=-w\rho-\rho $.

\newcommand{\Wsmall}{W_{[\lambda]}}

The  Verma module of $ \mf{g} $ with highest weight $ \lambda\in\hs $ is defined to be \[
M(\lambda):=U(\mf{g})\otimes_{U(\mf{b}^+)}\mathbb{C}_\lambda
\]where $ \mathbb{C}_\lambda $ is a $U(\mf{b}^+) $-module such that \[
h.m=\lambda(h)m,~~\mf{n}^+.m=0, \text{ for } h\in\hs,m\in\mathbb{C}_\lambda.
\]The Verma module $ M(\lambda) $ has a unique simple quotient, which is denoted by $ L(\lambda) $. These modules $ L(\lambda) $ with $ \lambda\in\hs $ are called the simple highest weight $ \mf{g} $-modules.

Let $ \Lambda=\{\lambda\in\hs\mid\bil{\lambda}{\al}\in\mathbb{Z}\text{ for all }\al\in\Phi^+ \} $, whose elements are called \textit{integral} weights. A weight $ \lambda\in\hs $ is called \textit{antidominant} if $ \bil{\lambda+\rho}{\al} \notin\mathbb{Z}_{>0}$ for all $ \al\in\Phi^+ $. Let   \begin{equation}\label{eq:phil}
\Phi_\lambda:=\{\al\in\Phi\mid s_\al.\lambda-\lambda\in\mathbb{Z}\al \}=\{\al\in\Phi\mid \bil{\lambda+\rho}{\al}\in\mathbb{Z} \},
\end{equation} 
and let $ \Wsmall $ be the subgroup of $ W $ generated by $ \{s_\al\mid \al\in\Phi_\lambda \} $. 
Note that $ \Wsmall $ is still a Coxeter group, but in general it is not  generated by a subset of $ S $. There is one and only one antidominant weight in $ \Wsmall.\lambda $. 

%

The category $ \msc{O} $ consists of $ \mf{g} $-modules which are \begin{itemize}
\item semisimple as $ \mf{h} $-modules;
\item finitely generated as $ U(\mf{g}) $-modules;
\item locally $ \mf{n} $-finite.
\end{itemize}
It is known that  $ \msc{O} $ is an abelian category which is both notherian and artinian, and has enough projective objects and injective objects. The simple objects of $ \msc{O} $ are precisely those $ L(\lambda) $ with $ \lambda\in\hs $.
A \textit{block} of $ \msc{O} $ is an indecomposable summand of $ \Ocat $ as an abelian subcategory. Let $ \Ocat_\lambda $  be the block containing  the simple module $ L(\lambda) $. Then we have $ \Ocat_\lambda=\Ocat_\mu $ if and only if $ \mu\in W_{[\lambda]}.\lambda $; equivalently, the simple modules in $ \Ocat_\lambda $ are precisely $ L(\mu) $, $ \mu\in\Wsmall.\lambda $. Since there is a unique antidominant  weight in $ \Wsmall.\lambda$, we can assume that $ \lambda $ is antidominant when considering a block $ \Ocat_\lambda $. By definition we have\[
\Ocat=\bigoplus_{\lambda\text{ antidominant}}\Ocat_\lambda.
\]
When $ \lambda=0 $, $ \Ocat_\lambda=\Ocat_0 $ is called the principal block. The simple modules in $ \Ocat_0 $  are denoted by \[
L_w:=L(w.(-2\rho)), w\in W.
\]

In what follows, we recall Jantzen's translation functors which are used to compare two blocks. Here we only consider the case of  integral weights, which is enough for the purpose of this paper.

Let $ E:= \mathbb{Z}\Phi\otimes_\mathbb{Z}\mathbb{R} $. It is a Euclidean space, on which the bilinear form  is just the restriction to $ E $ of the bilinear form $ (\,,\,) $ on $ \mf{h}^* $ defined above. One can see that $ \Lambda\subset E $. The Weyl group $ W $ also acts on $ E $ by dot action. Consider the hyperplanes $ H_\al=\{\lambda\in E\mid\bil{\lambda+\rho}{\al} =0\} $, $ \al\in\Phi^+ $.
These hyperplanes give rise to some facets. Precisely, a \textit{facet} $ F $ is a maximal subset of $ E $ such that for any $ \lambda,\mu\in F $, $ \bil{\lambda+\rho}{\al}=0 $ (resp. $ <0 $, or $ >0 $) if and only if $ \bil{\mu+\rho}{\al}=0 $ (resp. $ <0 $, or	 $ >0 $) for all $ \al\in\Phi^+ $.
The facets of maximal dimension are the connected components of $ E\setminus\bigcup_{\al\in\Phi^+}H_\al $, which are called  chambers. Denote by $ C^+ $ (resp. $ C^- $) the chamber consisting of $ \lambda $ such that $ \bil{\lambda+\rho}{\al} >0$ (resp. $ <0 $) for all $ \al\in\Phi^+ $.

For a facet $ F $ with $ \lambda\in F $, let $ \Phi_{>0}(F) =\{\al\in\Phi^+\mid \bil{\lambda+\rho}{\al} >0 \}$, $ \Phi_{=0}(F) =\{\al\in\Phi^+\mid \bil{\lambda+\rho}{\al} =0 \}$, $ \Phi_{<0}(F) =\{\al\in\Phi^+\mid \bil{\lambda+\rho}{\al} <0 \}$. The \textit{upper closure} of $ F $ is defined to be the set of $ \mu $ such that \[
\bil{\mu+\rho}{\al}>0 \text{ for }\al\in\Phi_{>0}(F),
\]
\[
\bil{\mu+\rho}{\al}=0 \text{ for }\al\in\Phi_{=0}(F),
\]\[\text{and }
\bil{\mu+\rho}{\al}\leq0 \text{ for }\al\in\Phi_{<0}(F).
\]
For example, the upper closure of $C^-  $ is the closure $ \overline{C^-} $, and the upper closure of $ C^+ $ is $ C^+ $ itself.

\begin{Lem}\label{lem:min-char}
The space $ E $ is the disjoint union of upper closures of chambers. For any integral weight $ \lambda\in\Lambda $, there is a unique $ w\in W $ such that $ \lambda $ is in the upper closure of the chamber $ w.C^- $ containing $ w.(-2\rho) $. Moreover, the element $ w $ here is characterized as the unique element of $ W $ of minimal length such that $ w^{-1}.\lambda$ is antidominant. Generally, for any weight $ \lambda\in \hs $, there is a unique element $ w \in \Wsmall $ of minimal length  such that $ w^{-1}.\lambda $ is antidominant.
\end{Lem}

For $ \lambda,\mu \in\Lambda$, we set $ \gamma=\mu-\lambda $. We can find a weight $ \bar{\gamma} \in W\gamma$ such that
$ \bil{\bar{\gamma}}{\al} \geq0$ for all $ \al \in\Phi^+$. Then $ L(\bar{\gamma}) $ is finite dimensional. The Jantzen's \textit{translation functor} $$  T_\lambda^\mu :\Ocat_\lambda\rar\Ocat_\mu $$ is   an exact functor given by $ T_\lambda^\mu(M):={\mr Pr}_\mu (L(\bar{\gamma})\otimes M) $, where $ M\in\msc{O}_\lambda $ and $ {\mr Pr}_\mu $ is the natural projection $ \Ocat\rar\Ocat_\mu $.
From \cite[\S2.10-2.11]{Jantzen79Lec750}, we have the following lemma.

\begin{Lem}\label{lem:translation}
Let $ \lambda,\mu\in\Lambda $.
\begin{itemize}
\item [(i)] $ T_\lambda^\mu $ is biadjoint to $ T_\mu^\lambda $.

\item [(ii)] If $ \lambda\in\Lambda $ is in a facet $ F $ and $ \mu \in\Lambda$ is in the closure $ \bar{F} $, then \[
T_{\lambda}^\mu(L(\lambda))\cong\begin{cases}
L(\mu),& \text{ if }\mu\text{ is in the upper closure of }F;\\
0,&\text{ otherwise}.
\end{cases}
\]
\end{itemize}
\end{Lem}

\begin{Cor}\label{cor:translation}
Let $ \lambda\in\Lambda $. We have\[
T_{w.(-2\rho)}^\lambda (L_w)=L(\lambda)
\]where $ w\in W $ is the unique element of minimal length such that $ w^{-1}.\lambda $ is antidominant.
\end{Cor}
\begin{proof}
By Lemma \ref{lem:min-char}, $ \lambda $ is in the upper closure of the facet that contains $ w.(-2\rho) $. Then the corollary follows from Lemma \ref{lem:translation} (ii).
  \end{proof}

\begin{Lem}\label{lem:upper}
Let $ \lambda\in\Lambda $ be an element of a facet $ F $ and $ \mu \in\Lambda$ is in the upper closure of $ F $, then\[
\gkd L(\lambda)=\gkd L(\mu).
\]
\end{Lem}
\begin{proof}
By   definition, we have $ \gkd T_\lambda^\mu L(\lambda)\leq \gkd L(\lambda)$ and $ \gkd T^\lambda_\mu L(\mu)\leq \gkd L(\mu)$.

By Lemma \ref{lem:translation} (ii), $ \gkd L(\mu)=\gkd T_\lambda^\mu L(\lambda)\leq \gkd L(\lambda) $. By  Lemma \ref{lem:translation} (i), we have $ \hom(L(\lambda),T_\mu^\lambda L(\mu) )=\hom(T_\lambda^\mu L(\lambda),L(\mu ))=\hom (L(\mu),L(\mu))\neq0$, which implies that $ \gkd L(\lambda)\leq \gkd T_\mu^\lambda L(\mu) \leq \gkd L(\mu)$. Then the lemma follows.
  \end{proof}

\subsection{Lusztig's \af-function}\label{sec:cell}

Recall that the Weyl group $ W  $ of $ \mf{g} $ is a Coxeter group generated by $ S=\{s_\al\mid\al\in\Delta \} $. Then we have a Hecke algebra $ \mc{H} $ over $ \mc{A} :=\mathbb{Z}[v,v^{-1}]$, which is generated by $ T_w $, $ w\in W $ with relations \[
T_wT_{w'}=T_{ww'} \text{ if }l(ww')=l(w)+l(w'),
\]\[
\text{and }(T_s+v^{-1})(T_s-v)=0 \text{ for any }s\in S.
\]

If $ \lambda\in\hs $, associated to the Weyl group $ \Wsmall $ of $ \mf{g}_\lambda $, we have a Hecke algebra $ \mc{H}_{[\lambda]} $. One need to note that in general $ \mc{H}_{[\lambda]} $ is not a subalgebra of $  \mc{H}$, since $ \Wsmall$ is not necessarily generated by a subset of $ S $.

The Kazhdan-Lusztig basis $ C_w $, $ w\in W $  of $ \mc{H} $ are characterized as the unique elements $ C_w $ such that\[
\overline{C_w}=C_w,\qquad C_w\equiv T_w \mod{\mc{H}_{<0}}
\]where $ \bar{\,} :\mc{H}\rar\mc{H}$ is the bar involution such that $ \bar{q}=q^{-1} $, $ \overline{T_w} =T_{w^{-1}}^{-1}$, and $ \mc{H}_{<0}=\bigoplus_{w\in W}\mc{A}_{<0}T_w $, $ \mc{A}_{<0}=v^{-1}\mathbb{Z}[v^{-1}] $.

If $ C_y $ occurs in the expansion of $ hC_w $ with respect to the KL-basis for some $ h\in\mc{H} $, then we write $ y\leftarrow_L w $. And we extend $ \leftarrow_L $ to a preorder $ \prec_L $ on $ W $. Define the equivalence relation $ \sim_L $ by that $ x\sim_L y $ if and only if $ x\prec_L y $ and $ y\prec_L x $. An equivalence class of $ \sim_L $ on $ W $ is called a  \textit{left cell} of $ W $. We define $ x\prec_R y $ by $ x^{-1}\prec_L y^{-1} $ and define $ x\sim_{R} y $ by $ x^{-1}\sim_L y^{-1} $. The equivalence class of $ \sim_R $ is called a \textit{right cell}. Define $ \prec_{LR} $ to be the preorder generated by $ \prec_L $ and $ \prec_R $. And we have  an equivalence relation $ \sim_{LR} $ on $ W $ corresponding to $ \prec_{LR} $; a   equivalence class of $ \sim_{LR} $ is called a two-sided cell of $ W $.

Let $ C_xC_y=\sum_{z\in W} h_{x,y,z}C_z $ with $ h_{x,y,x}\in\mc{A} $. Then  \textit{Lusztig's \af-function} $ \mathbf{a}:W\rar\mathbb{N} $ is defined by\[
\aff(z)=\max\{\deg h_{x,y,z}\mid x,y\in W \} \text{ for } z\in W.
\]

We   state a lemma which will be used later.

\begin{Lem}\label{lem:Hecke}
\begin{itemize}
\item [(i)] The equivalence relation $ \sim_{LR} $ is    generated by $ \sim_L $ and $ \sim_R $.
\item [(ii)]$ \aff(w)=\aff(w^{-1}) $ for all $ w\in W $.
\item [(iii)] $ \aff:W\rar\mathbb{N} $ is constant on each two-sided cell of  $ W $.
\item [(iv)] If $ w_I $ is the longest element of the parabolic subgroup of $ W $ generated by a subset $ I\subset S $, the $ \aff(w_I)$ is equal to  the length $l(w_I) $ of $ w_I $.
\item [(v)] If $ W $ is a direct product of  Coxeter subgroups $ W_1 $ and $ W_2 $, then\[
\aff(w)=\aff(w_1)+\aff(w_2)
\]for  $ w=(w_1,w_2) \in W_1\times W_2=W$.
\end{itemize}
\end{Lem}
\begin{proof}
We use the properties   P1-P15 in \S14.2 of \cite{lusztig2003hecke}. The statement (i) follows from P9-P11, (iii) follows from P4, (iv) follows from P12, and (ii),(v) are deduced from the definition.
  \end{proof}

For $ \lambda\in\hs $ the \af-function of $ \Wsmall $ is denoted by \[
\aff_{[\lambda]}:\Wsmall\rar\mathbb{N},
\]
which is defined using the Hecke algebra $ \mc{H}_{[\lambda]} $. One need to note that $\aff_{[\lambda]} $ is not the restriction of $ \aff:W\rar\mathbb{N} $.

\subsection{Gelfand-Kirillov dimensions of   simple highest weight modules}

By \cite[\S1]{lusztig1985A_n}, there is a formula connecting the \gk and the Lusztig's \af-function. Recall that $ L_w $ is the simple $ \mf{g} $-module in $ \msc{O}_0 $ with highest weight $ w.(-2\rho) $.
\begin{Lem}[Lusztig]\label{lem:tie}The \gk of $ L_w $ is given by
\[
\gkd L_w=\nu_{0}-\aff(w)\text{ for }w\in W,
\]where $ \nu_{0}=|\Phi^+| $ is the number of positive roots of $ \mf{g} $.
\end{Lem}


\begin{Defff}\label{def:a_lambda}
For $ \lambda\in \hs $, we define \[ \aff(\lambda): =\aff_{[\lambda]} (w)\] where $ w $ is the element of $ \Wsmall $ of minimal length such that $ w^{-1}.\lambda $ is antidominant.
\end{Defff}

\begin{Prop}\label{pr:main1}
For any $ \lambda\in\hs $, we have\[
\gkd L(\lambda)=\nu_{0}-\aff(\lambda).
\]
\end{Prop}
 \begin{proof}
First, we assume that   $ \lambda\in\Lambda $. In this case, $ W_{[\lambda]}=W $.	By Corollary \ref{cor:translation}, Lemma \ref{lem:upper} and Lemma \ref{lem:tie}, we have
\begin{equation}\label{eq:qq2}
\gkd L(\lambda)\overset{\ref{cor:translation}+\ref{lem:upper}}{=}\gkd L(w.(-2\rho))\overset{\ref{lem:tie}}{=}|\Phi^+|-\aff(w),
\end{equation}
where $ w$ is the element of $W $ of minimal length such that $ w^{-1}.\lambda $ is antidominant. Then we see that  the proposition holds for integral weights.

Now we return to the general  assumption that $ \lambda\in\mf{h}^* $.  In the following, we will reduce the proof  to the integral weight case.

Let $ \mf{g}_\lambda $ be a reductive Lie algebra such that its Cartan subalgebra coincides with  the Cartan subalgebra $ \mf{h} $ of $ \mf{g} $ and its root system  is given by $ \Phi_\lambda $ (see \eqref{eq:phil}). Then $ W_{[\lambda]} $ is the Weyl group of $ \mf{g}_\lambda $. Note that in general we can not realize $ \mf{g} _\lambda$ as a Lie subalgebra of $ \mf{g} $ unless $ \mf{g} $ is of $ ADE $-type.

We denote by $ \ms{O}^0 $ the category $ \ms{O} $ for  $ \mf{g}_\lambda $ with respect to $ \Phi_\lambda^+:=\Phi_\lambda\cap \Phi^+ $, and denote by $ L^0(\mu) $ (resp. $ M^0(\mu) $) the simple  module (resp. Verma modules) in $ \ms{O}^0  $ with highest weight $ \mu \in \mf{h}^*$. Let  $ \ms{O}^0_\lambda $ be the block of  $ \ms{O}^0 $ that containing $ L^0(\lambda) $. All the simple modules (resp. Verma module) in $ \ms{O}^0_\lambda $ are given by $ L^0(w.\lambda) $ (resp. $ M^0(w.\lambda) $) with $ w\in \Wsmall $. By \cite[Thm.11]{Soergel1990}, we have an equivalence of categories
\begin{equation}\label{eq:equi}
\ms{O}^0_\lambda\simeq\ms{O}_\lambda
\end{equation}
under which $ L^0(w.\lambda) $ (resp. $ M^0(w.\lambda) $) corresponds to $ L(w.\lambda) $ (resp. $ M(w.\lambda) $).

Recall that for a semisimple  $ \mf{h} $-module $ M=\bigoplus_{\lambda\in\mf{h}^*}M_\lambda $ with $ M_{\lambda}=\{m\in M\mid h.m=\lambda(h)m \} $ and $ \dim M_\lambda<\infty $, the character of $ M $ is defined to be a formal sum $ \mr{ch} M=\sum_{\mu\in \mf{h}^*}(\dim M_\mu )e^\mu $. For more details about the characters of   modules in $ \ms{O} $,   see \cite[\S1.15]{Hum08}. For $ w\in W_{[\lambda]} $, we have
\[
\mr{ch}M(w.\lambda)=\dfrac{e^{w.\lambda}}{\prod _{\alpha\in\Phi^+}(1-e^{-\alpha})},
\qquad
\mr{ch}M^0(w.\lambda)=\dfrac{e^{w.\lambda}}{\prod _{\alpha\in\Phi_\lambda^+}(1-e^{-\alpha})}.
\]
Hence, $ \mr{ch}M(w.\lambda)=\dfrac{1}{\prod_{\alpha\in\Phi^+\setminus\Phi^+_\lambda}(1-e^{-\alpha})}\mr{ch}M^0(w.\lambda) $. From the equivalence \eqref{eq:equi}, we see that   $ \mr{ch}L(w.\lambda) =\sum_{y\in W_{[\lambda]}} a_{y,w} \mr{ch} M(y.\lambda)$ for some integers $ a_{y,w} $ if and only if $ \mr{ch}L^0(w.\lambda)=\sum_{y\in W_{[\lambda]}} a_{y,w} \mr{ch} M^0(y.\lambda)$.
Thus, \begin{equation*}\label{eq:chL}
\mr{ch}L(w.\lambda)=\dfrac{1}{\prod_{\alpha\in\Phi^+\setminus\Phi^+_\lambda}(1-e^{-\alpha})}\mr{ch}L^0(w.\lambda).
\end{equation*}
This implies that\footnote{This implication is not obvious.  It can be proved in the following way. The enveloping algebra $ U(\mathfrak{g}) $ is filtered by the subspaces $ U_n(\mathfrak{g}) $ which is spanned by the monomials of degrees less than or equal to $ n $. It is known that $ gr U(\mathfrak{g})=\bigoplus_{n\geq1}U_n(\mathfrak{g})/U_{n-1}(\mathfrak{g}) $ is isomorphic to the symmetric algebra $ S(\mathfrak{g}) $. Let $ \mu\in\mathfrak{h}^* $, and $ v_\mu\in L(\mu) $ be a highest weight vector. Then $ gr L(\mu) :=\bigoplus_{n\geq1} U_{n}(\mathfrak{g}).v_\mu/U_{n-1}(\mathfrak{g}).v_\mu $ is an $ S(\mathfrak{g}) $-module generated by $ v_\mu $. Since $ S(\mathfrak{g}) $ is commutative, we have $ \mathfrak{n}^+.gr L(\mu)=S(\mathfrak{g})\mathfrak{n}^+.v_\mu=0$.
	
Since $ \mathfrak{g}_\mu $ is a subspace of $ \mathfrak{g} $, we have an embedding $ S(\mathfrak{g}_\mu)\hookrightarrow S(\mathfrak{g}) $. One can see that there is a surjective map  $ S(\mathfrak{g})\otimes_{S(\mathfrak{g}_\mu+\mathfrak{n}^+)}gr L^0(\mu)\twoheadrightarrow  gr  L(\mu) $ of $ S(\mathfrak{g}) $-algebras, where the $ S(\mathfrak{g}_\mu) $-module $ gr L^0(\mu) $ is extended to an $ S(\mathfrak{g}_\mu+\mathfrak{n}^+) $-module by letting $ \mathfrak{n}^+.gr L^0(\mu)=0 $. This map is actually an isomorphism  since
$ \mr{ch}L(\mu)=\dfrac{1}{\prod_{\alpha\in\Phi^+\setminus\Phi^+_\mu}(1-e^{-\alpha})}\mr{ch}L^0(\mu)$. Then by the definition of GK dimension, we have $ \gkd gr L(\mu)=|\Phi^+|-|\Phi_\mu^{+}|+\gkd gr L^0(\mu) $.  Now  \eqref{eq:qq1} follows from \cite[Prop. 6.6]{Kr-Le}. (Note that in general we have no map $ U(\mathfrak{g})\otimes_{U(\mathfrak{g}_\mu+\mathfrak{n}^+)} L^0(\mu)\to   L(\mu) $, since $ \mathfrak{g}_\mu $ is not necessarily a Lie subalgebra of $ \mathfrak{g} $.)
} 
\begin{equation}\label{eq:qq1}
 \gkd L(w.\lambda)=|\Phi^+|-|\Phi_\lambda^{+}|+\gkd L^0(w.\lambda).
\end{equation}
 By the definition of $ \mf{g}_\lambda $,  $ \lambda $ is integral with respect to  $ \mf{g}_\lambda $. Then by the first part of the proof, we have
 \[
 \gkd L^0(\lambda)=|\Phi_ \lambda^+|-\aff(\lambda).
 \]
 Combined with \eqref{eq:qq1}, we obtain that
 \[
 \gkd L(\lambda)=|\Phi^+|-\aff(\lambda).
 \]
Then the proposition follows.
   \end{proof}

\section{Gelfand-Kirillov dimensions for $ \mf{sl}(n) $}

\subsection{Schensted insertion algorithm}\label{sec:schensted}
A Young diagram is a collection of   boxes arranged in left-justified rows, with the row lengths weakly increasing. Let $ \Gamma $ be a totally ordered set; in the situation that we will encounter in this paper,  $ \Gamma$ is taken to be $c+\mathbb{N} $ for some $ c\in\mathbb{C} $. A \textit{Young tableau} $ Y $ is a Young diagram with each box filled with an element of  $ \Gamma $. A Young tableau is called \textit{standard} if the entries in each row and in each column are strictly  increasing. And a Young tableau  is called \textit{semistandard} if the entries in each row (resp. column) are weakly (resp. strictly) increasing.

Let $ \gamma=(\gamma_1,\cdots,\gamma_n) $ be a sequence of elements in $ \Gamma $. The following Schensted insertion algorithm associates $ \gamma $ to a pair $( P(\gamma) $, $ Q(\gamma)) $ of  semistandard  Young tableaux.

Let $ P_0 $, $ Q_0 $ be two empty Young tableaux. Assume that we have constructed Young tableaux $ P_k $, $ Q_k$ associated to $ (\gamma_1,\cdots,\gamma_k) $, $ 0\leq k<n $. Then $ P_{k+1} $ is obtained by adding $ \gamma_{k+1} $ to $ P_k $ as follows. First add $ \gamma_{k+1} $ to the first row of $ P_k $ by replacing the leftmost entry $ x $ in the first row which is \textit{strictly} bigger than $ \gamma_{k+1} $.  (If there is no such an entry $ x $, we just add a box with entry $\gamma_{k+1}  $ to the right side of the first row, and end this process.) Then add $ x $ to the next row as the same way of adding $ \gamma_{k+1} $ to the first row.  The Young $ Q_{k+1} $ is obtained by adding a box with entry $ k+1 $ to $ Q_k $ at the position of  $ P_{k+1}\setminus P_k $. Then we put $P(\gamma)=P_n  $ and $ Q(\gamma)=Q_{n} $. Note that $ P(\gamma) $ is a semistandard Young tableau and $ Q(\gamma) $ is a standard Young tableau, and they have the same shape.

 For example, if $ \gamma=(3,5,2,2,1) $, then the Young tableaux  produced by this algorithm are:
\[
\young(3)\rar\young(35)\rar\young(25,3)\rar \young(22,35)\rar\young(12,25,3)=P(\gamma),\]\[
\young(1)\rar\young(12)\rar\young(12,3)\rar \young(12,34)\rar\young(12,34,5)=Q(\gamma).
\]
\subsection{Lusztig's \af-function for $ \mf{g}=\mf{sl}(n) $}\label{sec:notation}
In this section, $ \mf{g} $ is assumed to be the special linear Lie algebra $ \sln $ and the Cartan subalgebra $ \mf{h} $ is taken to be the subalgebra  consists of all diagonal matrices in $ \sln $. Let $ \epsilon_i $, $ i\in [1,n] $ be the elements of $ \hs $ such that $ \epsilon_i(E_{jj})=\delta_{ij} $ where $ E_{ij} $ is the matrix unit with entry 1 at $ (i,j) $-position. We have a relation $ \sum_{i=1}^{n}\epsilon_i=0 $ and we have $ \hs=(\bigoplus_{i=1}^{n}\mathbb{C}\epsilon_i)/\mathbb{C}(\sum_{i=1}^n\epsilon_i) $. The set of positive roots of $ \sln $ is taken to be $ \Phi^+=\{\epsilon_i-\epsilon_j\mid i<j \} $. The set of simple roots  is  $ \Delta=\{\epsilon_i-\epsilon_{i+1}\mid i<n \} $. The number of positive roots is $ \nu_{0}=\frac{n(n-1)}{2} $.

For a weight $ \lambda\in \hs $, we write\[
\lambda+\rho=(\lambda_1,\lambda_2,\cdots,\lambda_n),
\]where $ \lambda+\rho=\sum_{i=1}^{n}\lambda_i\epsilon_i $, $ \lambda_i\in\mathbb{C} $. Since $ \sum_{i=1}^{n}\epsilon_i =0$ we have $ (\lambda_1+c,\lambda_2+c,\cdots,\lambda_n+c)=(\lambda_1,\lambda_2,\cdots,\lambda_n) $ for any $ c\in \mathbb{C} $.
For example, $ 2\rho=(n-1,n-3,\cdots, -n+1) $ and $ \rho=(-1,-2,\cdots,-n) $.

Let $ \mf{S} _n$ be the symmetric group in $ n $ letters $ 1,2,3,\cdots,n $. Then $ \mf{S} _n$ can be identified with the Weyl group of $ \sln $ via the action\[
\sigma(\epsilon_i)=\epsilon_{\sigma(i)}, i\in[1,n], \sigma\in\mf{S}_n.
\]Then simple reflections of $ \sn $ with respect to $ \Delta $ are $ (i,i+1) $, $1\leq i<n $. For an element $ \sigma\in\sn $, we use the notation $ \sigma=\begin{pmatrix}
1 & 2 & \cdots & n \\
\sigma_1 & \sigma_2 & \cdots & \sigma_n
\end{pmatrix}  $ where $ \sigma_i=\sigma(i) $, $ i\in[1,n] $. We denote by $ P(\sigma) $ and $ Q(\sigma) $ the standard  Young tableaux associated to the sequence $ (\sigma_1, \sigma_2,\cdots,\sigma_n) $, which are produced by the Schensted insertion algorithm in \S\ref{sec:schensted}. For example, if $ \sigma=\begin{pmatrix}
1 & 2 & 3 & 4 \\
3 & 4& 2&1
\end{pmatrix}   $ then $ P(\sigma)=\young(14,2,3) $ and $ Q(\sigma)=\young(12,3,4) $.
The well-known Robinson-Schensted correspondence says that there is a one-to-one correspondence between $ \sn $ and the set of pairs of standard  Young tableaux (filled with 1, 2, \dots, n) of the same shape.
Moreover, we have
\begin{Lem}\label{lem:rs}
\begin{itemize}
\item [(i)] $ P(\sigma)=Q(\sigma^{-1}) $ for any $ \sigma\in\sn $.
\item [(ii)] $ \sigma,\tau\in\sn $ are in the same right (resp. left ) cell if and only if $ P(\sigma)=P(\tau) $ (resp. $  Q(\sigma)=Q(\tau)$); see \S\ref{sec:cell} for the definition of cells.
\item [(iii)]$ \sigma,\tau\in\sn $ are in the same two-sided  cell if and only if $ P(\sigma)$, $P(\tau) $ have the same shape.
\end{itemize}
\end{Lem}
\begin{proof}
(i) is straightforward from the definition. (ii) follows from \cite[Thm. A]{Ariki2000} and (i). The `` only if " part of (iii) follows from (i) of Lemma \ref{lem:Hecke} and (ii) of this lemma. And the `` if " part follows from the `` only if " part, the following Lemma \ref{lem:zy} and the known fact that the number of two-sided cells of $ \sn $ is equal to that of partitions of $ n $.
  \end{proof}
\begin{Lem}\label{lem:zy}
Let $ Y $ be a Young tableau with $ c_i $ entries in the $ i $-th column, and  $ \sum_i c_i=n $. Let $ \sigma_Y $ be the longest element in the parabolic subgroup of $ \mf{S}_n $ generated by $ s_k=(k,k+1) $, $ k\in[1,n]\setminus\{\sum_{j=1}^{i} c_j\mid i\geq1 \} $. Then $ P(\sigma_Y) $ and $ Y $ have the same shape. More precisely, the entries in the $ i $-th column of $ P(\sigma_Y) $ are  the integers $ x $ such that $ \sum_{j=1}^{i-1} c_j<x \leq  \sum_{j=1}^{i} c_j $.
\end{Lem}
\begin{proof}
This is straightforward from the algorithm for $ P(\sigma) $.
  \end{proof}

\begin{Defff}\label{def:zay}
Let $ Y $ be a Young tableau with $ c_i $ entries in the $ i $-th column.  Define $ A(Y) $ to be the integer $ \sum_{i\geq 1} \frac{c_i(c_i-1)}2$, which is the length of  the element $ \sigma_Y $ in the above lemma.
\end{Defff}

\begin{Prop}\label{pr:avalue}
For $ \sigma\in\sn $, we have \[
\aff(\sigma)=A(P(\sigma)).
\]
\end{Prop}
\begin{proof}
Let $ Y=P(\sigma) $. Since $ P(\sigma_Y) $ and $ P(\sigma) $ are of the same shape (see Lemma \ref{lem:zy}),  $ \sigma_Y $ and $ \sigma $ are in the same two-sided cell by Lemma \ref{lem:rs}(iii). Hence $ \aff(\sigma)=\aff(\sigma_Y)=l(\sigma_Y)= A(Y)$ by Lemma \ref{lem:Hecke} (iii), (iv) and the Definition \ref{def:zay}.
  \end{proof}
 This proposition gives an efficient algorithm for Lusztig's function $ \aff:\sn\rar\mathbb{N} $. In next subsection we will give a similar algorithm for $ \aff(\lambda) $, $ \lambda\in\hs  $ (see Definition \ref{def:a_lambda}).

\subsection{ Algorithm for $ \aff(\lambda) $}
Recall that for $\lambda\in\hs  $ we write $ \lambdarho$. One can see that $ \lambda $ is an integral weight  if and only if $ \lambda_i-\lambda_j\in \mathbb{Z} $ for all $ i,j $, and $ \lambda $ is antidominant if and only if $ \lambda_i\leq \lambda_j $ whenever $ \lambda_i-\lambda_j\in \mathbb{Z} $, $ i<j $. If $ \sigma\in\sn $, then $ \sigma.\lambda+\rho=\sigma(\lambda+\rho)=\sum_i\lambda_i\epsilon_{\sigma(i)}=\sum_{j}\lambda_{\sigma^{-1}(j)}\epsilon_j $.

If $ \lambda\in\Lambda $ is an integral weight, we associated  to the sequence  $ (\lambda_1,\cdots,\lambda_n) $ an semistandard Young tableau $ P(\lambda) $; see \S\ref{sec:schensted}.
\begin{Lem}Assume that $ \lambda $ is an integral weight.\label{lem:zzaA}
\begin{itemize}
\item [(i)] There is a unique $\sigma_\lambda\in\sn $ such that for $ i<j $\begin{equation}\label{eq:zz1}
\lambda_i\leq \lambda_j \text{ if and only if } \sigma_\lambda(i)<\sigma_\lambda(j),
\end{equation}
\begin{equation}\label{eq:zz2}
\text{ and }\lambda_i> \lambda_j \text{ if and only if } \sigma_\lambda(i)>\sigma_\lambda(j).
\end{equation}
\item [(ii)]The element $ \sigma_\lambda\in\sn $ is  of minimal length such that $ \sigma_\lambda.\lambda  $ is antidominant.
\item [(iii)]$ P(\lambda) $ and $ P(\sigma_\lambda) $ have the same shape. Hence \begin{equation}
\aff(\lambda) =A(P(\lambda));
\end{equation}see Definition \ref{def:zay} for the definition of the function $ A $.
\end{itemize}
\end{Lem}
\begin{proof}
(i) is obvious. Let $ \sigma=\sigma_\lambda $. Then $ \sigma.\lambda+\rho=(\lambda_{\sigma^{-1}(1)},\lambda_{\sigma^{-1}(2)},\cdots,\lambda_{\sigma^{-1}(n)}) $. If  $ a=\sigma(p)<\sigma(q)=b $ then by \eqref{eq:zz1}, \eqref{eq:zz2} we always have $ \lambda_p\leq \lambda_q$, i.e. $\lambda_{\sigma^{-1}(a)}\leq\lambda_{\sigma^{-1}(b)}  $. Thus $ \sigma.\lambda $ is antidominant. To prove $ \sigma $ is of minimal length we only need to prove that when $ s_\al.\lambda=\lambda $ with $ \al\in\Phi^+ $ we have $ \sigma s_\al>\sigma $, or equivalently that if $ \lambda_i=\lambda_j $, $ i<j $ then $ \sigma(\epsilon_i-\epsilon_j) \in \Phi^+$, i.e.  $ \sigma(i)<\sigma(j) $. But this is just \eqref{eq:zz1}.

(iii). By the construction of $ \sigma_\lambda $ and the Schensted insertion algorithm, we see immediately that $ P(\lambda) $ and $ P(\sigma_\lambda) $ have the same shape.  Then\[
\aff(\lambda)\overset{\text{Def. } \ref{def:a_lambda}}{=} \aff(\sigma_\lambda^{-1})\overset{\text{Lem. \ref{lem:Hecke}(ii)}}{=} \aff(\sigma_\lambda) \overset{\text{Prop. \ref{pr:avalue}}}{=}A(P(\sigma_\lambda))=A(P(\lambda)).
\]
 This completes the proof.
  \end{proof}

Now we turn to the general case where $ \lambda $ is not necessarily integral.

We fix a weight $ \lambda\in\hs $. Denote by $ \mf{X}_\lambda $ or just $ \mf{X} $ the set of subsets $ X $ of $ [1,n] $ such that $ i,j $ are in $ X $ if and only if $ \lambda_i-\lambda_j\in\mathbb{Z} $. If $ X=\{i_1\leq i_2\leq\cdots\leq i_k \} $ then denote by $ \lambda_X$ the weight such that $\lambda_X+\rho' =(\lambda_{i_1},\lambda_{i_2},\cdots,\lambda_{i_k})$, where $ \lambda_X $  is viewed as an integral weight of $ \mf{sl}(k) $ and $ \rho' $ is the half sum of positive roots of $ \mf{sl}(k) $.

Then we associate  to $ \lambda $ the set $ P(\lambda) $, which consists of all the Young tableaux $ P(\lambda_X) $ with $ X\in\mf{X} $. And we define $ A(P(\lambda)) $ to be $ \sum_{X\in\mf{X}}A(P(\lambda_X)) $.

\begin{Thm}\label{thm:main2}
Assume that $ \mf{g}=\mf{sl}(n) $, $ \lambda\in\hs $. Then \[
\aff(\lambda)=A(P(\lambda)).
\]Hence by Proposition \ref{pr:main1}, we have\[
\gkd L(\lambda)=\frac{n(n-1)}2-A(P(\lambda)).
\]This gives rise to  a combinatorial algorithm for $ \gkd L(\lambda) $.
\end{Thm}
\begin{proof}
First we have \[
\Wsmall=\prod_{X\in\mf{X}} W_X
\]where $ W_X $ is the subgroup of $ W=\sn $ generated by $ (i,j) $ with $ i,j\in X $. Then $ w $ is of minimal length in $ \Wsmall $ such that $ w^{-1}.\lambda $ is antidominant if and only if $ w_X $ is of minimal length in $ W_X $ such that $ w_X^{-1}.\lambda_X $ is antidominant, where $ w=(w_X)_{X\in\mf{X}} $. Then we have
\begin{align*}
	\aff(\lambda)&=\aff_{[\lambda]}(w)\overset{\text{Lem. \ref{lem:Hecke}(v)}}{=}\sum_{X\in\mf{X}} \aff_{[\lambda_X]}(w_X)\\&=\sum_{X\in\mf{X}}\aff(\lambda_X)\overset{\text{Lem. \ref{lem:zzaA}}}{=}\sum_{X\in\mf{X}}A(P(\lambda_X))=A(P(\lambda)).
\end{align*}
This completes the proof of the theorem.
  \end{proof}

\section{$ \gkd L(\lambda) $ for $ (p,q) $-dominant weights}

Let $ \mf{g} =\mf{sl(n)}$, $ p+q=n $ with $ p,q\in\mathbb{Z}_{\geq1} $.
\begin{Defff}
We say $ \lambda \in\hs$ is \textit{$ (p,q) $-dominant} if
  $ \lambda_i-\lambda_j\in\mathbb{Z}_{>0} $ for all $ i,j  $ such that \(
  1\leq i<j\leq p \text{ or } p+1\leq i<j\leq p+q
\),
 where $ \lambdarho $. In particular, $ \lambda_1>\lambda_2>\cdots>\lambda_p $ and $ \lambda_{p+1}>\lambda_{p+2}>\cdots>\lambda_{p+q} $.
\end{Defff}

The results in this section give rise to a quick algorithm for $ \gkd L(\lambda) $ when $ \lambda $ is a  $  (p,q) $-dominant weight.

\begin{Thm}\label{pr:pqw}
Assume that $ \lambda\in\hs $ is $ (p,q) $-dominant. \begin{itemize}
\item [(i)] If $ \lambda_1-\lambda_{p+1}\in\mathbb{Z} $, i.e. $ \lambda $ is an integral weight, then $ P(\lambda) $ is a Young tableau with at most two columns. And in this case $ \gkd L(\lambda)=m(n-m) $ where $ m $ is the number of entries in the second column of $ P(\lambda) $.
\item [(ii)] If $ \lambda_1-\lambda_{p+1}\notin\mathbb{Z} $, then $ P(\lambda) $ consists of two Young tableaux with   single column:
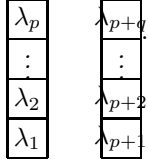
\[\tableau{\lambda_p\\\vdots\\\lambda_2\\\lambda_1}\qquad \tableau{\lambda_{p+q}\\\vdots\\\lambda_{p+2}\\\lambda_{p+1}}.\]
And in this case $ \gkd L(\lambda)=pq $.
\end{itemize}
\end{Thm}
\begin{proof}
The shape of $ P(\lambda) $ follows directly from the definition. And the \gk follows from Theorem \ref{thm:main2}.
  \end{proof}

\subsection{An algorithm for $ m $}By the above theorem, we only need to determine $ m $ in the case where $ \lambda $ is integral and $ (p,q) $-dominant.
\begin{Prop}\label{pr:main3}
Assume that  $ \lambda $ is an integral and $ (p,q) $-dominant weight. Then we have the following three cases.

\textit{Case (i)} $ \lambda_{p+1}\geq\lambda_p>\lambda_{p+q} $. Then we can find an integer $ 1\leq k\leq q $ such that $ \lambda_p\leq\lambda_{p+k} $ and $ \lambda_p>\lambda_{p+k+i} $ for all $ i>0 $. Let $ \lambda' $ be the weight corresponding to  $ (\lambda_1,\lambda_2,\cdots,\lambda_{p-1},\lambda_{p+1},\cdots,\lambda_{p+k-1}) $, i.e. deleting the entries $ \lambda_p$, $ \lambda_{p+k+j}, j\geq0 $.
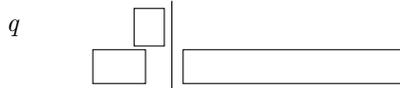
\begin{center}
\begin{tikzpicture}(0,0)
\put(0,0){$\lambda_1$}
\put(15,0){$ \lambda_2 $}
\put(30,0){$ \cdots $}
\put(50,0){$ \lambda_p $}
\put(-20,-15){$ \lambda_{p+1} $}
\put(5,-15){$ \cdots $}
\put(20,-15){$ \cdots $}
\put(35,-15){$ \lambda_{p+k} $}
\put(70,-15){$ \lambda_{p+k+1} $}
\put(100,-15){$ \cdots $}
\put(130,-15){$ \lambda_{p+q} $};
\draw(2.2,.4)--(2.2,-0.8);
\draw (1.7,-.2) rectangle (2.1,.3);
\draw (1.15,-.7) rectangle (1.85,-.25);
\draw (2.35,-.7) rectangle (5.3,-.25);
\end{tikzpicture}
\end{center}
Then $ P(\lambda) $ is obtained from $ P(\lambda') $ in the following way.
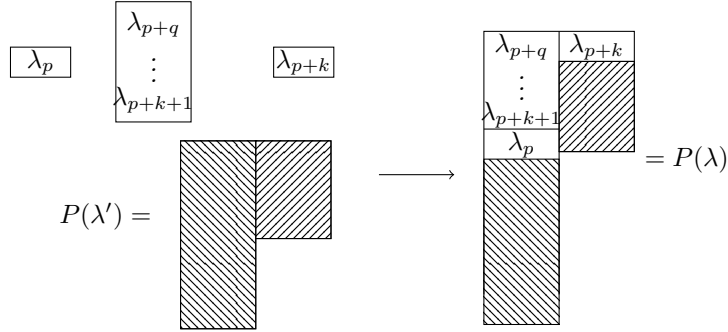
\begin{center}
\begin{tikzpicture}
\node (a) at (-0.5,1.5) {\begin{tikzpicture}
\draw (0,0) rectangle (1,1.6) ;
\node at (0.5,1.3) {$
\lambda_{p+q} $};

\node at (0.5,.8) {$
\vdots  $};

\node at (0.5,.3) {$
\lambda_{p+k+1}$};
\end{tikzpicture}};
\node (b) at (1.5,1.5) {\begin{tikzpicture}
\draw (0,0) rectangle (.8,.4);\node at (0.4,0.2) {$\lambda_{p+k}$};
\end{tikzpicture}};

\node (b) at (-2,1.5) {\begin{tikzpicture}
\draw (0,0) rectangle (.8,.4);\node at (0.4,0.2) {$\lambda_p$};
\end{tikzpicture}};

\node (c) at (0,-0.8) {\begin{tikzpicture}
\draw  (0,-1.5) -- (0,1)--(2,1)--(2,-0.3)--(1,-0.3)--(1,-1.5)--(0,-1.5);\node at(-1,0){$ P(\lambda')= $};
\draw[pattern=north west lines] (0,-1.5) rectangle (1,1);\draw[pattern=north east lines] (1,-.3) rectangle (2,1);
\end{tikzpicture}};\draw [->] (2.5,0) -- (3.5,0);
\node  at (5.5,0){\begin{tikzpicture}
\draw (0,0)--(0,3.5)--(2,3.5)--(2,1.9)--(1,1.9)--(1,0);\draw(1,1.9)--(1,3.5);\draw(0,2.2)--(1,2.2);\draw(1,3.1)--(2,3.1);\node (a) at (0.5,2.39) {$
\lambda_{p+k+1}$};
\node (a) at (0.5,3.3) {$
\lambda_{p+q} $};
\node (a) at (0.5,2.85) {$
\vdots
$};\node  at (1.5,3.3) {$
\lambda_{p+k}$};
\node  at (.5,1.99) {$
\lambda_{p}$};\node  at (2.7,1.8) {$
=P(\lambda)$};\draw[pattern=north west lines] (0,-0.4) rectangle (1,1.8);
\draw[pattern=north east lines] (1,1.9) rectangle (2,3.1);
\end{tikzpicture}};
\end{tikzpicture}
\end{center}
In other words, the second column of $ P(\lambda) $ can be obtained from that of $ P(\lambda') $ by adding a box containing $ \lambda_{p+k} $ on the top.

\textit{Case (ii)} 	 $ \lambda_p\leq \lambda_{p+q} $. Let $ \lambda' $ be the weight corresponding to  $$ (\lambda_1,\lambda_2,\cdots,\lambda_{p-1},\lambda_{p+1},\cdots,\lambda_{p+q-1}) , $$ i.e.   deleting the entries $ \lambda_p$, $ \lambda_{p+q}$ from $ \lambda $. Then $ P(\lambda) $ is obtained from $ P(\lambda') $ in the following way.
\begin{center}
\begin{tikzpicture}
%
%
\node (b) at (1.5,1.5) {\begin{tikzpicture}
\draw (0,0) rectangle (.8,.4);\node at (0.4,0.2) {$\lambda_{p+q}$};
\end{tikzpicture}};

\node (b) at (0,1.5) {\begin{tikzpicture}
\draw (0,0) rectangle (.8,.4);\node at (0.4,0.2) {$\lambda_p$};
\end{tikzpicture}};

\node (c) at (0,-0.8) {\begin{tikzpicture}
\draw  (0,-1.5) -- (0,1)--(2,1)--(2,-0.3)--(1,-0.3)--(1,-1.5)--(0,-1.5);\node at(-1,0){$ P(\lambda')= $};
\draw[pattern=north west lines] (0,-1.5) rectangle (1,1);\draw[pattern=north east lines] (1,-.3) rectangle (2,1);
\end{tikzpicture}};\draw [->] (2.5,0) -- (3.5,0);
\node  at (5.2,-1.5){\begin{tikzpicture}
\begin{tikzpicture}
\draw  (0,-1.5) -- (0,1.4)--(2,1.4)--(2,-0.3)--(1,-0.3)--(1,-1.5)--(0,-1.5);\draw (1,1.4)--(1,1);
\draw[pattern=north west lines] (0,-1.5) rectangle (1,1);\draw[pattern=north east lines] (1,-.3) rectangle (2,1);\node at (0.5,1.2) {$ \lambda_p $};\node at (1.5,1.2) {$ \lambda_{p+q} $};
\end{tikzpicture}\node  at (.7,1.5) {$
=P(\lambda)$};
\end{tikzpicture}};
\end{tikzpicture}
\end{center}
In other words, the second column of $ P(\lambda) $ is  obtained from that of $ P(\lambda') $ by adding a box containing  $ \lambda_{p+q} $ on the top.

\textit{Case (iii)} $ \lambda_p\geq \lambda_{p+1} $. Then $ P(\lambda) $ is just a Young tableau with single column.
\end{Prop}

\begin{proof}
This directly follows from Schensted insertion algorithm (see \S\ref{sec:schensted}).
  \end{proof}

This proposition gives rise to an algorithm for $ P(\lambda) $ in the following way.  First, we construct a sequence of weights $ \lambda(0),\lambda(1),\cdots,\lambda(l) $ for some $ l $. Here $ \lambda(0)=\lambda $. Assume that we have obtained $ \lambda(i) $  in some step. If the entries corresponding to $ \lambda(i) $ are not strictly decreasing, then we set $ \lambda(i+1)=\lambda(i)' $ as we do in Proposition \ref{pr:main3}(i)(ii). Otherwise, we set $ \lambda(i) $ to be the last weight $ \lambda(l) $.

Note that $ P(\lambda(l)) $ is an empty Young tableau or a Young tableau with single column. Then one can obtain   $ P(\lambda(l-1)) $ from $ P(\lambda(l)) $ by using (i) or (ii) of Proposition \ref{pr:main3}. Repeating this process, we get $ P(\lambda(i)) $ from $ P(\lambda(i+1)) $ for all $ i $, in particular  $ P(\lambda) =P(\lambda(0))$. From the process we know  that the number  $m  $ of entries in the second column of $ P(\lambda) $ is just $ l $ and $m\leq min(p,q)$. Then $ \gkd L(\lambda) $ also follows from this algorithm by using Theorem \ref{pr:pqw}(i). 

In fact, we do not need to write down the each $ P(\lambda(i)) $, since we can 
  immediately write down  the second column of  $ P(\lambda)  $ from $ \lambda(i) $'s; see the following example.

\begin{ex}\label{ex:alg}
Let  $ n=10, p=4, q=6 $ and $ \lambda+\rho=(6,5,3,2,9,8,7,4,2,1) $.

\begin{center}
\begin{tikzpicture}
\node at(0.1,0){\begin{tikzpicture} \node at (0,0){\begin{tikzpicture}
\node at (0,0){{ }};
\node at (.5,-.5){$ 9 $};
\node at (1,-.5){$ \underline{8} $};
\node at (1.5,-.5){$ \underline{7} $};
\node at (2,0){$ 6 $};
\node at (2.5,0){$ 5 $};
\node at (3,-.5){$ \underline{4} $};
\node at (3.5,0){$ 3 $};
\node at (4,0){$ 2 $};
\node at (4,-.5){$ \underline{2} $};
\node at (4.5,-.5){$ 1 $};
\end{tikzpicture}};
\end{tikzpicture}};
\draw 
(1.7,.5)--(1.7,-.5);
\draw
(.65,.5)--(.65,-.5);
\draw
(-.05,.5)--(-.05,0)--(-.95,0)--(-.95,-.5);
\draw
(-.55,.5)--(-.55,.07)--(-1.45,.07)--(-1.45,-.5);\end{tikzpicture}
\end{center}
Using the above notation, we have  \begin{eqnarray*}
\lambda(1)&\rar&(6,5,3,9,8,7,4),\\
\lambda(2)&\rar&(6,5,9,8,7),\\
\lambda(3)&\rar&(6,9,8),\\
\lambda(4)&\rar&(9).
\end{eqnarray*}
 One can  see that the second column of $ P(\lambda) $ is $ \{8,7,4,2 \} $. Hence we have $$  P(\lambda)=\young(12,24,37,58,6,9)\qquad \gkd L(\lambda)=4\times7=24. $$
\end{ex}

\subsection{A combinatorial model}
Let $ \lambda $ be an integral and $ (p,q) $-dominant weight such that $ \lambdarho $.
Let $ a_1 $ be the number of $ \lambda_{p+i},i\geq1 $ such that $ \lambda_{p+i}\geq\lambda_1 $ and let $ b_1 $ be the number of $ \lambda_j,j\geq1 $ such that $ \lambda_j>\lambda_{p+a_1+1} $. Inductively, let $ a_{k+1} $ be the number  such that \[
a_{k+1}=\begin{cases}
 \#\{\lambda_{p+i} \mid \lambda_{\sum_{t=1}^{k}b_t}>\lambda_{p+i}\geq\lambda_{(\sum_{t=1}^{k}b_t)+1} \}, &\text{ if }(\sum_{t=1}^{k}b_t)+1\leq p;\\
 \#\{\lambda_{p+i} \mid \lambda_{\sum_{t=1}^{k}b_t}>\lambda_{p+i} \}, &\text{ if }\sum_{t=1}^{k}b_t=p;\\
 0, &\text{ if }\sum_{t=1}^{k}b_t> p.
\end{cases}
\]  and $ b_{k+1} $ the number such that \[b_{k+1}=\begin{cases}
\#\{\lambda_{j}\mid \lambda_{p+\sum_{t=1}^{k+1}a_t}\geq\lambda_j>\lambda_{p+(\sum_{t=1}^{k+1}a_t)+1} \} & \text{ if }(\sum_{t=1}^{k+1}a_t)+1\leq q;\\
\#\{\lambda_{j}\mid \lambda_{p+\sum_{t=1}^{k+1}a_t}\geq\lambda_j \} & \text{ if }\sum_{t=1}^{k+1}a_t= q;\\
0 & \text{ if }\sum_{t=1}^{k+1}a_t> q.
\end{cases}
\]Let $ r\in\mathbb{Z}_{\geq1} $ be the minimal integer such that  $ a_{r+1}=b_{r+1} =0$. Then we write\[
\xi(\lambda):=(a_1,b_1,a_2,b_2,\cdots,a_r,b_r).
\]
Note that $ a_i $ and $ b_i $ are positive except that $ a_1,b_{r} $ may be $ 0 $. For the weight $ \lambda $ in Example \ref{ex:alg}, we have $ \xi(\lambda) =(3,2,1,1,1,1,1,0)$.

From the definition of  $ \xi(\lambda) $ we have  the following  lemma.

\begin{Lem}\label{lem:xi}
If $ \lambda,\mu $ are integral and $ (p,q) $-dominant weights satisfying $ \xi(\lambda)=\xi(\mu) $, then $ P(\lambda) $ and $ P(\mu) $ have the same shape. In particular, $ \gkd L(\lambda)=\gkd L(\mu) $.
\end{Lem}
\begin{proof}
This follows from Proposition \ref{pr:main3}.
  \end{proof}

The rest of this section is devoted to a combinatorial model, which enables us to obtain   $ m $   directly from $ \xi(\lambda) $.

Assume that we have $ p $ black balls and $ q $ white balls, which are arranged in a line such that from left to right there are $ a_1 $ white balls, $ b_1 $ black balls, $ a_2 $ white balls, \dots, $ b_r $ black balls. We call two balls an \textit{adjacent white-black pair} if they are adjacent with the left ball white and the right ball  black. We remove the \textit{adjacent} white-black pairs from this line of balls over and over again until there are no adjacent white-black  pairs, i.e the remaining black balls are all in the left side of every white ball. One can check that the number of pairs that are removed in this process are independent of the order, which is denoted by $ G_r $.

\begin{Thm}\label{pr:main4}
Let $ \lambda $ be an integral $ (p,q) $-dominant weight such that \[ \xi(\lambda)=(a_1,b_1,\cdots, a_r,b_r).\] Let $ m $ be the number of entries in the second column of $ P(\lambda) $. And let $ G_{k}, 1\leq k\leq r $ be the number of adjacent white-black pairs in the sequence $ (a_1,b_1,\cdots, a_k,b_k) $ of balls, described as above.
\begin{itemize}
\item [(i)] We have $m=G_r $. Hence $ \gkd L(\lambda) =G_r(n-G_r)$ by Theorem \ref{pr:pqw}.
\item [(ii)] There is a recursive formula \[
G_{k+1}=G_k+\min\{\sum_{i=1}^{k+1}a_i-G_k, b_{k+1} \}, \text{ for }k<r.
\]
\end{itemize}
\end{Thm}
\begin{proof}(i) follows from Proposition \ref{pr:main3}.

(ii). After deleting all adjacent white-black pairs in the sequence $(a_1,b_1,\cdots, a_k,b_k)  $, there are $ \sum_{=1}^{k}b_i-G_k $ black balls on the left and $ \sum_{i=1}^{k}a_i-G_k $ white balls on the right. If we take $ a_{k+1},b_{k+1} $ into account, then from left to right  we have $ \sum_{i=1}^{k}b_i-G_k $ black balls, $ \sum_{i=1}^{k+1}a_i-G_k $ white balls, and $ b_{k+1} $ black balls. Thus we in addition obtain $ \min \{\sum_{i=1}^{k+1}a_i-G_k,b_{k+1}\} $ adjacent white-black pairs. Then (ii) follows.
  \end{proof}

For the latter use, we formulate the following lemma.
\begin{Lem}\label{lem:transform}
Let $ \xi $ be a sequence of white and black balls as above. Then, locally, we can do the following two kinds of operations on $ \xi $
\begin{center}
 \begin{tikzpicture}\node at (0,0) 
 {\begin{tikzpicture}
 \node at ( 1,1)  [circle,draw=black]{};

 \node at ( 1.5,1) [circle,draw=black,fill=black!60]  {};\node at ( 2,1) [circle,draw=black,fill=black!60]  {};
 \end{tikzpicture}};\draw[<->] (1,0)--(2,0);\node at (3,0) 
  {\begin{tikzpicture}
  \node at ( 1,1) [circle,draw=black,fill=black!60] {};
 
  \node at ( 1.5,1) [circle,draw=black] {};\node at ( 2,1) [circle,draw=black,fill=black!60] {};
  \end{tikzpicture}};
  
  \node at (0,-.7) 
   {\begin{tikzpicture}
   \node at ( 1,1) [circle,draw=black]{};
  
   \node at ( 1.5,1)[circle,draw=black]  {};\node at ( 2,1) [circle,draw=black,fill=black!60]  {};
   \end{tikzpicture}};\draw[<->] (1,-.7)--(2,-.7);\node at (3,-.7) 
    {\begin{tikzpicture}
    \node at ( 1,1) [circle,draw=black] {};
   
    \node at ( 1.5,1) [circle,draw=black,fill=black!60]  {};\node at ( 2,1) [circle,draw=black]  {};
    \end{tikzpicture}};
 \end{tikzpicture}
 \end{center}
such that the number of adjacent white-black pairs in $ \xi $ does not change  under these operations.
\end{Lem}
\begin{proof}
It   is an easy consequence of the fact that the number of adjacent white-black pairs is independent of the choice of deleting order. 
\end{proof}

If $ \lambda $ is as in  Example \ref{ex:alg}, the sequence $ \xi(\lambda) $ of balls is illustrated as follows:
\begin{center}
\begin{tikzpicture}

\node at ( 1.5,1) [circle,draw=black] {};\node at ( 2,1) [circle,draw=black] {};\node at ( 2.5,1) [circle,draw=black] {};

\node at ( 3,1) [circle,draw=black,fill=black!60] {};
\node at ( 3.5,1) [circle,draw=black,fill=black!60] {};

\node at ( 4,1) [circle,draw=black] {};

\node at ( 4.5,1) [circle,draw=black,fill=black!60] {};

\node at ( 5,1) [circle,draw=black] {};

\node at ( 5.5,1) [circle,draw=black,fill=black!60] {};

\node at ( 6,1) [circle,draw=black] {};

\draw (2.5,0.8)--(3,.8);

\draw (2,0.7)--(3.5,.7);

\draw (4,0.8)--(4.5,.8);

\draw (5,0.8)--(5.5,.8);
\end{tikzpicture}
\end{center}
 Then we can remove 4 adjacent white-black pairs. And the remaining balls are arranged as follows: \begin{center}
 \begin{tikzpicture}

 \node at ( 1.5,1) [circle,draw=black] {};\node at ( 2,1) [circle,draw=black] {};
 \end{tikzpicture}
 \end{center}
So $ m=G_r=4 $. Using certain operations in Lemma \ref{lem:transform}, $ \xi(\lambda) $  can transform into 
\begin{center}
\begin{tikzpicture}

\node at ( 1.5,1) [circle,draw=black] {};
\node at ( 2,1) [circle,draw=black,fill=black!60] {};
\node at ( 2.5,1) [circle,draw=black] {};

\node at ( 3,1) [circle,draw=black,fill=black!60] {};
\node at ( 3.5,1) [circle,draw=black]  {};

\node at ( 4,1) [circle,draw=black,fill=black!60] {};

\node at ( 4.5,1)[circle,draw=black] {};

\node at ( 5,1) [circle,draw=black,fill=black!60]  {};

\node at ( 5.5,1)[circle,draw=black] {};

\node at ( 6,1) [circle,draw=black] {};
\end{tikzpicture}
\end{center}

\subsection{Another model}
At last, we want to give an algebraic model to compute $ m $.

\begin{Prop}\label{pr:alg}
Let $ A $ be a $ \mathbb{Z}[v] $-algebra generated by $ x,y $ with only relation \(
xy=v\)
where $ v $ is an indeterminate. It is obvious that $ A $ has a standard $ \mathbb{Z}[v] $-basis $\{y^sx^t\mid s,t\in\mathbb{N} \}  $.

Let $ \lambda $ be an integral $ (p,q) $-dominant weight with $ \xi(\lambda)=(a_1,b_1,\cdots,a_r,b_r) $. Then we have\[
x^{a_1}y^{b_1}\cdots x^{a_r}y^{b_r}=v^m y^{p-m}x^{q-m}
\]where $ m $ is the number of entries in the second column of $ P(\lambda) $.
\end{Prop}
\begin{proof}
This  directly follows from Theorem \ref{pr:main4} by viewing white (resp. black) balls as $ x $ (resp. $ y $).
  \end{proof}

Applying this proposition to Example \ref{ex:alg}, we have $ x^3y^2xyxyx=v^my^{4-m}x^{6-m} $, which implies that $ m=4 $.

\section{Application to $ (\mathfrak{g},K) $-modules}

\subsection{Application to  unitary highest weight $ \su(p,q) $-modules}

The  unitary highest weight $ (\mathfrak{g},K) $-modules had been classified by   \cite{EHW} and   \cite{Jac}. Let $L(\lambda) $ be a unitary highest weight $(\mathfrak{g}, K)$-module. We use  notations from   \cite{EHW}. Let $\beta$ be the unique maximal root in the positive roots $\Phi^+$ and $\zeta$ be the unique weight orthogonal to $ \Phi_c $  and satisfying $(\zeta, \beta^{\vee})=1$. Then  we can write $\lambda=\tilde{\lambda}+z \zeta $, where $\tilde{\lambda}$ is a $\mathfrak{k}$-dominant weight of $ \mathfrak{h}^{*} $ such that ($\tilde{\lambda}  + \rho, \beta  $)=0, and $z\in \mathbb{R}$. A formula for the Gelfand-Kirillov dimensions of all unitary highest weight  $ (\mathfrak{g},K) $-modules is  found in \cite{Bai-Hu}.  Now by our algorithm, we can reprove it in the case $G=\su(p,q)$.

Assume that  $ \tilde{\lambda} $ is a $ (p,q) $-dominant weight such that $ \tilde{\lambda}_1 =\tilde{\lambda}_{p+q}$ where $ \tilde{\lambda}+\rho=(\tilde{\lambda}_1,\cdots,\tilde{\lambda}_n) $. Without loss of generality, we can assume that all $ \tilde{\lambda}_i\in\mathbb{Z} $. We have $ \zeta=({\underbrace{1,1\cdots 1}_{p},\underbrace{0\cdots,0}_{q}}) $. Let $ p' $ be the maximal integer in $ [1,p] $ such that $ \tilde{\lambda}_1 $, $ \tilde{\lambda}_2 $, $ \cdots $, $ \tilde{\lambda}_{p'} $ are consecutive integers.
And let $ q' $ be the maximal integer in $ [1,q] $ such that $ \tilde{\lambda}_{p+q-q'+1} $,  $ \cdots $, $ \tilde{\lambda}_{p+q} $ are consecutive integers.

From \cite[Thm. 7.4]{EHW}, we have the following lemma.
\begin{Lem}Let $ z\in\mathbb{R} $. The simple highest weight module
	$ L(\tilde{\lambda}+z\zeta) $ is a unitary $ \su(p,q) $-module if and only if $ z\in I_{\tilde{\lambda}} $, where \[
	I_{\tilde{\lambda}}=\{z\in\mathbb{R}\mid z\leq \max{\{p',q'\}} \}\cup\{z\in\mathbb{Z}\mid z\leq p'+q'-1 \}.
	\]
\end{Lem}

Then we have the following result.

\begin{Prop}\label{pro:uhw}
	Keep the notations as above. When $ z\in I_{\tilde{\lambda}} $, the \gk of $ L(\tilde{\lambda}+z\zeta) $ only depends on $ z $. Precisely,\begin{itemize}
		\item [(i)]If $ z\in I_{\tilde{\lambda}}\setminus\mathbb{Z} $, we have $ \gkd L(\tilde{\lambda}+z\zeta)=pq $.
		\item [(ii)] If $ z\in I_{\tilde{\lambda}}\cap\mathbb{Z} $, we have
		\[
		\gkd L(\tilde{\lambda}+z\zeta)=\begin{cases}
		pq,&\text{if } z<\max\{p,q\} ;\\
		(z+1)(n-z-1), & \text{if } \max\{p,q \}\leq z \leq p'+q'-1 .
		\end{cases}
		\]
	\end{itemize}
\end{Prop}\begin{proof}
(i) follows from Theorem \ref{pr:pqw}(ii).

(ii). Let $ \tilde{\mu}+\rho =(p,p-1,\cdots, 1,p+q-1,p+q-2,\cdots, p)$. By using operations in Lemma \ref{lem:transform} one can check that $ P(\tilde{\lambda}+z\zeta)$ and $P(\tilde{\mu}+z\zeta)$ have the same shape   when $ z\in I_{\tilde{\lambda}}\cap\mathbb{Z} $. Then by Theorem \ref{thm:main2}, we have $$  \gkd (\tilde{\lambda}+z\zeta) =\gkd(\tilde{\mu}+z\zeta) \text{ for all } z\in I_{\tilde{\lambda}}\cap\mathbb{Z}. $$ Then (ii) is reduced to the computation of  $ \gkd(\tilde{\mu}+z\zeta) $, which explains why   $\gkd L(\tilde{\lambda}+z\zeta) $ only depends on $ z $ if $ z\in I_{\tilde{\lambda}} \cap\mathbb{Z}$.
Now we compute
$ \gkd L(\tilde{\mu}+z\zeta) $ for all $ z\in\mathbb{Z}  $. Using Theorem \ref{pr:main4}, one can check that the number $ m $ of entries in  the second column of $ P(\tilde{\mu}+z\zeta) $ with $ z\in\mathbb{Z} $ is
\[
m =\begin{cases}
\min\{{p,q}\},&\text{ if }z<\max\{p,q\};\\
n-1-z&\text{ if }\max\{p,q\}\leq z\leq n-1;\\
0&\text{ if } z\geq n.

\end{cases}
\]
Hence by Theorem \ref{pr:pqw}, we have \[
\gkd L(\tilde{\mu}+z\zeta)=\begin{cases}
pq,&\text{ if }z<\max\{p,q\};\\
(z+1)(n-1-z)&\text{ if}\max\{p,q\}\leq z\leq n-1;\\
0&\text{ if } z\geq n.

\end{cases}
\]
which completes the proof.
  \end{proof}


\subsection{Application to  highest weight $ \su(p,q) $-modules}
It is proved 
in \cite{Bai-Hu} that $\gkd  L(\tilde{\lambda}+z\zeta)$  decreases as $z$ takes unitary points and  increases.
Now we  generalize  this result to all highest weight $(\mf{g},K)$-modules in the case $ G=\su(p,q) $. 

\begin{Thm}\label{th:last}
Let $ L(\tilde{\lambda}+z\zeta)$ be a highest weight $\su(p,q)$-module. 
Then $$ GKdim(L(\tilde{\lambda}+z\zeta)) $$ will (weakly) decrease   as $z\in \mathbb{Z}$ increases. Moreover, $GKdim(L(\tilde{\lambda}+z\zeta))=0$ if $ z $ is an integer bigger than  $ \tilde{\lambda}_{p+1}-\tilde{\lambda}_p$.
\end{Thm}

\begin{proof}   By the combinatorial model in Theorem \ref{pr:main4}, the value of $m$ for $P(\tilde{\lambda}+(z+1)\zeta)$ is less than or equal to the value of $m$ for $P(\tilde{\lambda}+z\zeta)$. So  $m$  decreases as the integer $z$ increases.
	Note that $m\leq \min(p,q)\leq [\frac{n}{2}]$.  Then  Theorem \ref{pr:pqw} implies that  $\gkd  L(\tilde{\lambda}+z\zeta)$ decreases as $ m $ decreases.  So  $\gkd  L(\tilde{\lambda}+z\zeta)$ decreases as  the integer $z$ increases.
  \end{proof}

Now we turn to consider associated varieties of highest weight $ (\mf{g},K) $-modules.
We have the decomposition $\mathfrak{g}=\mathfrak{p}^-+\mathfrak{k}+\mathfrak{p}^+$ under the adjoint action of $K_{\mathbb{C}}$.   From Vogan \cite{Vogan-91}, we know that the associated variety of any highest weight $(\mathfrak{g}, K)$-module $L(\lambda)$ is the closure of some $K_{\mathbb{C}}$ orbit in $\mathfrak{p}^+\cong (\mathfrak{g}/(\mathfrak{k}+\mathfrak{p}^+))^*$, and closures of these $K_{\mathbb{C}}$ orbits form a linear chain of varieties:
\begin{equation}
\{0\}={\bar{\mc{O}}}_0\subset \bar{\mathcal{O}}_1\subset ...\subset\bar{\mathcal{O}}_{r-1}\subset \bar{\mathcal{O}}_r=\mathfrak{p}^+,
\end{equation}
where $r$ is the rank of the Hermitian symmetric space $G/K$. In \cite{Bai-Hu}, it is proved that 
$$\dim\bar{\mathcal{O}}_{k}=k\bil{\rho}{\beta}-k(k-1)C,$$
where $C$ is a constant only depending on $G$.

In the case $ G=\su(p,q) $, we have  $ \bil{\rho}{\beta} =n-1$ and $  C=1$, and hence $  \dim\bar{\mathcal{O}}_{k} =k(n-k) .$
Since $\gkd L(\lambda)=m(n-m) $  where  $ m $ is the number of entries in the second column of $ P(\lambda) $, then the associated variety of $ L(\lambda) $ must be $ \bar{\mathcal{O}}_{m} $. To summarize, we have the following theorem about Gelfand-Kirillov  dimensions and associated varieties of a highest weight $ (\mf{g},K) $-modules in the case $ G=\su(p,q) $.

\begin{Thm}\label{th:last2} Let $ G=\su(p,q) $. Let $ L(\lambda)$ be a highest weight $(\mf{g},K)$-module. 
\begin{itemize}
\item [(i)] If  $\lambda$ is an integral weight, then \begin{itemize}
\item[(a)] $ P(\lambda) $ is a Young tableau with at most two columns;\item [(b)] $ \gkd L(\lambda)=m(n-m) $ where $ m $ is the number of entries in the second column of $ P(\lambda) $; \item[(c)] The associated variety of $ L(\lambda)$  is $\bar{\mathcal{O}}_m$.
\end{itemize}
\item [(ii)] If  $\lambda$ is not  integral, then \begin{itemize}
\item[(a)] $ P(\lambda) $ consists of two Young tableaux with single column;\item[(b)]  $ \gkd L(\lambda)=pq $;\item[(c)] The associated variety of $ L(\lambda)$  is $\bar{\mathcal{O}}_r$, $ r=\min\{p,q\} $.
\end{itemize}
\end{itemize}
\end{Thm}

\subsection*{Acknowledgments}
The first author is supported by NSFC Grant No.11601394 and China Postdoctoral Science Foundation Grant No.2016M602337. The second author is   supported by NSFC Grant No.11631001 and NSFC Grant No.11601116. Both authors thank the anonymous
referees for valuable comments and suggestions.

\end{document}